\documentclass[11pt]{amsart}

\usepackage{latexcad}

\newtheorem{theorem}{Theorem}[section]
\newtheorem{corollary}[theorem]{Corollary}
\newtheorem{lemma}[theorem]{Lemma}

\newtheorem{remark}[theorem]{Remark}
\newtheorem{example}[theorem]{Example}

\newtheorem{problem}[theorem]{Problem}

\def\ldiv{\setminus}
\def\rdiv{/}
\def\refover#1{\overset{\eqref{#1}}=}

\title[Equational Foundations for Loops]{A Scoop from Groups:
Equational Foundations for Loops}

\author{J.~D.~Phillips}

\address{Department of Mathematics \& Computer Science, Wabash College,
Crawfordsville, Indiana 47933, U.S.A.}

\email{phillipj@wabash.edu}

\author{Petr Vojt\v{e}chovsk\'y}

\address{Department of Mathematics, University of Denver, 2360 S Gaylord St,
Denver, CO, 80208, U.S.A.}

\email{petr@math.du.edu}

\begin{document}

\begin{abstract}
Groups are usually axiomatized as algebras with an associative binary
operation, a two-sided neutral element, and with two-sided inverses. We show in
this note that the same simplicity of axioms can be achieved for some of the
most important varieties of loops. In particular, we investigate loops of
Bol-Moufang type in the underlying variety of magmas with two-sided inverses,
and obtain ``group-like'' equational bases for Moufang, Bol and C-loops. We
also discuss the case when the inverses are only one-sided and/or the neutral
element is only one-sided.
\end{abstract}



\maketitle

\section{Magmas, semigroups, and loops}

%FIGURE
\setlength{\unitlength}{1.5mm}
\begin{figure}[ht]\begin{center}\input{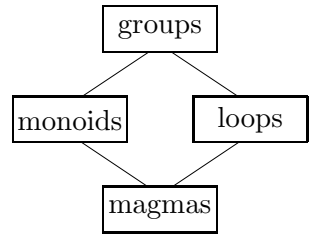}\end{center}
\caption{Two paths from magmas to groups.}\label{Fg:Quadrangle}
\end{figure}

We call a set with a single binary operation a \emph{groupoid}, and a groupoid
with a two-sided neutral element a \emph{magma}.\footnote{The definitions of
\emph{groupoid} and \emph{magma} are often interchanged in the literature.}
There are two natural paths from magmas to groups, as illustrated in Figure
\ref{Fg:Quadrangle}. One path leads through the \emph{monoids}---these are the
associative magmas. The other path leads through the \emph{loops}---these are
magmas in which every equation $x\cdot y=z$ has a unique solution whenever two
of the elements $x$, $y$, $z$ are specified. Since groups are precisely loops
that are also monoids, loops are known colloquially as ``nonassociative
groups.''

The theory of monoids is well-developed and widely known. Loops have a deep and often
elegant, but less well-known theory (see \cite{Pf2000} for
historical notes on loop theory, \cite{Br1971} for the first systematic
account of loops, and \cite{Pf1990} for a modern introduction to loop
theory). This is particularly
unfortunate given the deep and fruitful connections between loops and:
\begin{enumerate}

\item[(i)] combinatorics (Cayley tables of finite loops are normalized Latin squares \cite{DeKe};
Steiner loops describe Steiner triple systems \cite{CoRo}),

\item[(ii)] group theory (multiplication groups and automorphism groups of
loops often yield classical groups \cite{Do}, \cite{SpVe2000}; loops play a
role in the construction of the Monster sporadic group \cite{Co}),

\item[(iii)] division algebras (nonzero octonions under multiplication form a loop
\cite{SpVe2000}),

\item[(iv)] nonassociative algebras (alternative algebras \cite{GoMi}, Jordan algebras
\cite{Sc}),

\item[(iv)] projective geometry (generalized polygons, Moufang planes
\cite{TiWe2003}),

\item[(v)] special relativity (relativistic operations can be described by
loops \cite{Un}, \cite{Ki}).
\end{enumerate}

We believe that one of the reasons why loops are not more widely known is that
\emph{they cannot be defined equationally in the variety of magmas}, since they
are not closed under the taking of homomorphic images. This peculiar property
of loops resurfaces every now and then (most recently in \cite{Problem}), and
it was first observed by Bates and Kiokemeister \cite{BaKi}.

The standard way out of this impasse, due to Evans \cite{Ev}, is to introduce
two additional binary operations $\ldiv$, $\rdiv$, and demand that
\begin{equation}\label{Eq:Evans}
    x\cdot(x\ldiv y)=y,\quad (x\rdiv y)\cdot y = x,\quad
    (x\cdot y)\rdiv y = x,\quad x\ldiv(x\cdot y)=y,\quad
    x\rdiv x=x\ldiv x=1.
\end{equation}
Indeed, we obtain loops, since the axioms \eqref{Eq:Evans} imply that $x\ldiv
y$ is the unique solution $z$ to the equation $x\cdot z=y$, and similarly for
$x\rdiv y$.

While this approach solves the problem in principle, it is somewhat awkward. In
the end, the three operations $\cdot$, $\ldiv$, $\rdiv$ can be reconstructed
from any one of them!

The purpose of this note is to show that there is a much better solution for
some (but not all) of the most studied varieties of loops. We prove:

\begin{theorem}\label{Th:TwoSided}
Let $Q$ be a magma with two-sided inverses, that is, $1\cdot x = x\cdot 1 = x$
and $x\cdot x^{-1}=x^{-1}\cdot x=1$ holds for every $x\in Q$. If $Q$ satisfies
any of \eqref{Eq:LB}, \eqref{Eq:M1}, \eqref{Eq:M2}, \eqref{Eq:C} defined below,
then $Q$ is a loop.
\end{theorem}

This means that Bol loops, Moufang loops and C-loops can be axiomatized in a
manner completely analogous to groups.

We establish stronger (but perhaps less natural) results than Theorem
\ref{Th:TwoSided} upon looking at groupoids with a one-sided neutral element
and/or one-sided inverses.

\subsection{The Dot Convention}

We will write $xy$ instead of $x\cdot y$, and reserve $\cdot$ to indicate
parentheses and hence the priority of multiplication. For instance, $x\cdot
yz$ stands for $x\cdot (y\cdot z)$. This convention is common in
nonassociative algebra.

\section{The inverse property}

\noindent Every element $x$ of a magma $M$ determines two maps $M\to M$: the
\emph{left translation} $L_x:y\mapsto xy$, and the \emph{right translation}
$R_x:y\mapsto yx$. The equations $ax = b$, $ya=b$ have unique solutions $x$,
$y$ in $M$---that is, $M$ is a loop---if and only if all translations are
bijections of $M$.

We shall say that a magma $M$ is \emph{with inverses} (or that it \emph{has
inverses}) if for every $x\in M$ there is $y\in M$ satisfying $xy=yx=1$. We
then call $y$ an \emph{inverse} of $x$, noting that $x$ can have several
inverses.

We say that a magma $M$ has the \emph{left inverse property} if for every $x\in
M$ there is $x^\lambda\in M$ such that $x^\lambda\cdot xy=y$ for every $y\in
M$. Similarly, $M$ has the \emph{right inverse property} if for every $x\in M$
there is $x^\rho\in M$ such that $yx\cdot x^\rho=y$ for every $y\in M$. An
\emph{inverse property magma} is then a magma that has both the left inverse
property and the right inverse property.

\begin{lemma}\label{Lm:LIPMagmas}
If $M$ is a magma that satisfies the left inverse property, then $M$ is with
inverses, and the unique inverse of $x$ is $x^{-1}=x^\lambda$. Moreover,
$(x^{-1})^{-1}=x$, and all left translations are bijections of $M$.
\end{lemma}
\begin{proof}
For $x\in M$ we have $1=x^\lambda\cdot x1=x^\lambda x$. Then
$x=(x^\lambda)^\lambda\cdot x^\lambda x = (x^\lambda)^\lambda$ and
$xx^\lambda = (x^\lambda)^\lambda x^\lambda = 1$. Thus $x^\lambda=x^{-1}$ is
an inverse of $x$, and it is unique: if $xx^*=1$ for some $x^*$ then
$x^*=x^{-1}\cdot xx^* = x^{-1}$.

If $xy=xz$, then $y=x^{-1}\cdot xy=x^{-1}\cdot xz=z$. Furthermore, $x\cdot
x^{-1}y = (x^{-1})^{-1}\cdot x^{-1}y=y$. Thus $L_x$ is a bijection of $M$.
\end{proof}

We can now axiomatize inverse property loops in a manner analogous to groups:

\begin{theorem}\label{Th:IPMagmas}
Inverse property loops are exactly inverse property magmas and can be defined
equationally by
\begin{displaymath}
    x\cdot 1 = 1\cdot x = x,\quad x^\lambda\cdot xy=y=yx\cdot x^\rho,
\end{displaymath}
or by
\begin{displaymath}
    x\cdot 1 = 1\cdot x = x,\quad x^{-1}\cdot xy=y=yx\cdot x^{-1}.
\end{displaymath}
\end{theorem}
\begin{proof}
By the inverse properties, we have $1 = x^\lambda x$ and $1=1x\cdot x^\rho =
xx^\rho$. Then, $x^\lambda=x^\lambda1=x^\lambda\cdot xx^\rho = x^\rho$. We
are done by Lemma \ref{Lm:LIPMagmas} and its dual.
\end{proof}

\section{Inverses in loops of Bol-Moufang type}

%FIGURE
\setlength{\unitlength}{0.94mm}
\begin{figure}[ht]\begin{center}\input{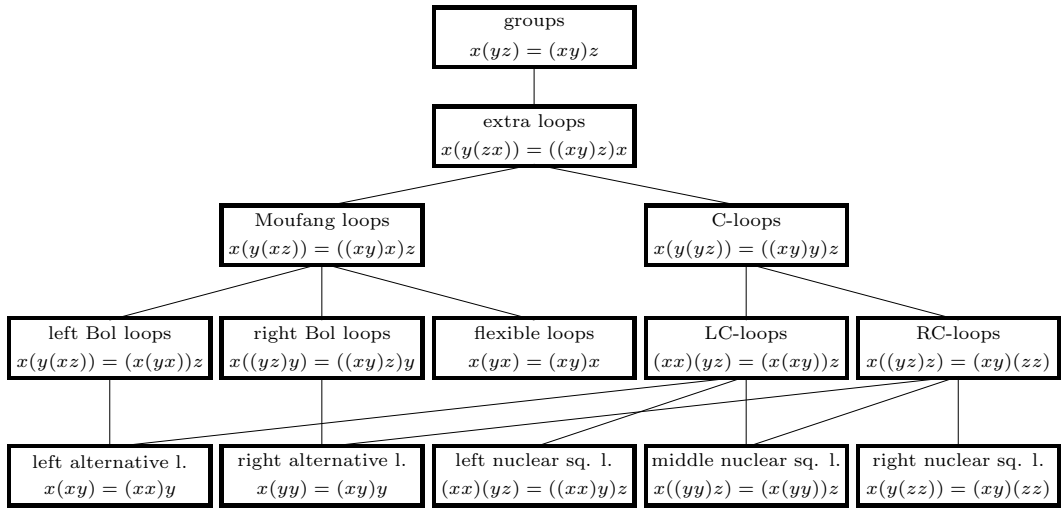}\end{center}
\caption{The varieties of loops of Bol-Moufang type.}\label{Fg:BM}
\end{figure}

\noindent Just as the theory of monoids focuses on those monoids satisfying
certain identities, so too for loops. Among the most investigated loops are the
so-called \emph{loops of Bol-Moufang type}, whose defining identities can be
found in Figure \ref{Fg:BM}. The figure also depicts all inclusions (but not
meets and joins) among varieties of loops of Bol-Moufang type. For more
details, see \cite{Fe} and \cite{PhVo2004}.

Since we are interested in these loops from the viewpoint of magmas with
inverses, let us first settle the question which varieties of loops of
Bol-Moufang type \emph{have} inverses.

By \cite{Ro}, left Bol loops satisfy the left inverse property. Dually, right
Bol loops satisfy the right inverse property. By \cite{Fe}, LC-loops satisfy
the left inverse property. Dually, RC-loops satisfy the right inverse property.
Lemma \ref{Lm:LIPMagmas} and its dual then imply that left (right) Bol loops
and LC(RC)-loops have inverses. Flexible loops also have inverses: if $x'$
satisfies $x'x=1$, then $(xx')x = x(x'x) = x$, and $xx'=1$ follows by
cancelation. On the other hand, consider
\begin{displaymath}
    Q_1 = \begin{array}{c|cccccc}
    &0&1&2&3&4&5\\
    \hline
    0&0&1&2&3&4&5\\
    1&1&5&0&4&2&3\\
    2&2&4&5&0&3&1\\
    3&3&0&4&5&1&2\\
    4&4&2&3&1&5&0\\
    5&5&3&1&2&0&4
    \end{array},\quad\quad
    Q_2 = \begin{array}{c|cccccc}
    &0&1&2&3&4&5\\
    \hline
    0&0&1&2&3&4&5\\
    1&1&5&0&4&2&3\\
    2&2&4&5&0&3&1\\
    3&3&0&4&5&1&2\\
    4&4&3&1&2&5&0\\
    5&5&2&3&1&0&4
    \end{array}.
\end{displaymath}
Then $Q_1$ is a left alternative loop without inverses, and $Q_2$ is a left,
middle, and right nuclear square loop without inverses. Hence the loops of
Bol-Moufang type with inverses occupy precisely the top four rows of Figure
\ref{Fg:BM}.

\begin{remark}
It is an open question, due to W.~D.~Smith \cite{Sm}, whether there is a
finite, left alternative and right alternative loop without inverses. For an
infinite example, see \cite{OrVo}.
\end{remark}

\section{Equational bases for Bol, Moufang, and C-loops}\label{Sc:Eq}

Now that we know which loops of Bol-Moufang type have inverses, we proceed to
obtain simple axiomatizations for three varieties: Bol, Moufang, and C-loops.

Let us label the \emph{alternative laws} by
\begin{equation}\label{Eq:LA}
    x\cdot xy=xx\cdot y,\tag{LA}
\end{equation}
and
\begin{equation}\label{Eq:RA}
    x\cdot yy=xy\cdot y.\tag{RA}
\end{equation}

\subsection{Bol loops}

Left Bol loops with the \emph{automorphic inverse property}
$(xy)^{-1}=x^{-1}y^{-1}$ play an important role in the arithmetic of special
relativity \cite{Ki}, \cite{Un}. Some of the most outstanding problems in loop
theory are concerned with Bol loops, and several of them were recently solved
by G.~P.~Nagy \cite{Na}.

Label the \emph{left Bol identity} as
\begin{equation}\label{Eq:LB}
    (x \cdot yx)z = x(y \cdot xz).\tag{LB}
\end{equation}

The following theorem and its generalizations have a convoluted history, cf.
\cite[pp.\ 50--51]{Ki}. It is the first result concerning a variety of loops of
Bol-Moufang type within the variety of magmas with inverses. We believe that it
was first observed by M.~K.~Kinyon, with a different proof:

\begin{theorem}\label{Th:Bol}
A magma with inverses satisfying the left Bol identity \eqref{Eq:LB} is a loop.
Thus, left Bol loops are defined equationally by
\begin{displaymath}
    1\cdot x = x\cdot 1 = x,\quad x\cdot x^{-1}=x^{-1}\cdot x = 1,
    \quad (x\cdot yx)z = x(y\cdot xz).
\end{displaymath}
\end{theorem}
\begin{proof}
Assume that $M$ is a magma with inverses satisfying \eqref{Eq:LB}. By setting
$y = 1$ in \eqref{Eq:LB} we see that the left alternative law holds for $M$.

We now show that $M$ has the left inverse property:
\begin{gather}
    x^{-1}\cdot x(x\cdot x^{-1}y) \refover{Eq:LA}
        x^{-1}(xx\cdot x^{-1}y) \refover{Eq:LB} (x^{-1}\cdot(xx)x^{-1})y
        \refover{Eq:LA} y,\label{Eq:AuxLB1}\\
    x(x^{-1}\cdot xy) \refover{Eq:LB} xy,\label{Eq:AuxLB2}\\
    x(x\cdot x^{-1}y) \refover{Eq:AuxLB2}
        x(x^{-1}\cdot x(x\cdot x^{-1}y)) \refover{Eq:AuxLB1}
        xy,\label{Eq:AuxLB3}\\
    x^{-1}\cdot xy \refover{Eq:AuxLB3}
        x^{-1}\cdot x(x\cdot x^{-1}y) \refover{Eq:AuxLB1} y.\nonumber
\end{gather}
By Lemma \ref{Lm:LIPMagmas}, all left translations are bijections of $M$. When
$xy=z$ then
\begin{displaymath}
    yz\cdot y^{-1} = (y\cdot xy)y^{-1} \refover{Eq:LB} yx,
\end{displaymath}
and thus $x = y^{-1}\cdot (yz)y^{-1}$. This means that the right translation
$R_y$ is a bijection.
\end{proof}

\subsection{Moufang loops}

These four \emph{Moufang identities} are equivalent for loops:
\begin{align}
    (xy\cdot x)z &= x(y\cdot xz),\label{Eq:M1}\tag{M1}\\
    x(y\cdot zy) &= (xy\cdot z)y,\label{Eq:M2}\tag{M2}\\
    xy\cdot zx &= x(yz\cdot x),\label{Eq:M3}\tag{M3}\\
    xy\cdot zx &= (x\cdot yz)x.\label{Eq:M4}\tag{M4}
\end{align}
Moufang loops occur naturally in division algebras and in projective geometry,
as we have already mentioned in the introduction.

Any of the first two identities can be used to characterize Moufang loops among
magmas with inverses:

\begin{theorem}\label{Th:Moufang}
A magma with inverses satisfying the Moufang identity \eqref{Eq:M1} or
\eqref{Eq:M2} is a Moufang loop. Thus, Moufang loops are defined equationally
by
\begin{displaymath}
    x\cdot 1 = 1\cdot x = x,\quad x\cdot x^{-1}=x^{-1}\cdot x = 1,
    \quad (xy\cdot x)z = x(y\cdot xz),
\end{displaymath}
or by
\begin{displaymath}
    x\cdot 1 = 1\cdot x = x,\quad x\cdot x^{-1}=x^{-1}\cdot x = 1,
    \quad x(y\cdot zy) = (xy\cdot z)y.
\end{displaymath}
\end{theorem}
\begin{proof}
Assume that $M$ is a magma with inverses satisfying \eqref{Eq:M1}. Substituting
$z=1$ into \eqref{Eq:M1} yields the flexible law $xy\cdot x = x\cdot yx$. Then
\eqref{Eq:M1} can be rewritten as \eqref{Eq:LB}, and Theorem \ref{Th:Bol} shows
that $M$ is a Moufang loop.

The case \eqref{Eq:M2} is similar (let $x=1$ in \eqref{Eq:M2}, and use the
right Bol identity).
\end{proof}

But the identities \eqref{Eq:M3}, \eqref{Eq:M4} do not work! Here is the Cayley
table of a magma with inverses that satisfies both \eqref{Eq:M3} and
\eqref{Eq:M4} but that clearly is not a loop:
\begin{displaymath}
\begin{array}{c|ccc}
      &0&1&2\\
     \hline
     0&0&1&2\\
     1&1&0&1\\
     2&2&1&0
\end{array}
\end{displaymath}

\subsection{C-loops}

Recall the \emph{C-identity}
\begin{equation}\label{Eq:C}
    x(y\cdot yz) = (xy\cdot y)z.\tag{C}
\end{equation}
C-loops were introduced by Fenyves \cite{Fe}. It is easy to see that every
Steiner loop (i.e., a loop arising from a Steiner triple system) is a C-loop
\cite{PV}. The standard Cayley-Dickson process extended beyond octonions
(dimension $8$) produces C-loops in every dimension $2^n$ \cite{KiPhVo}.
Although C-loops are not as well known as Bol and Moufang loops, we expect
their prominence to grow.

\begin{theorem}\label{Th:C} A magma with inverses satisfying the C-identity
\eqref{Eq:C} is a C-loop. Thus, C-loops are defined equationally by
\begin{displaymath}
    x\cdot 1 = 1\cdot x = x,\quad x\cdot x^{-1}=x^{-1}\cdot x = 1,
    \quad x(y\cdot yz) = (xy\cdot y)z.
\end{displaymath}
\end{theorem}
\begin{proof}
First note that a magma satisfying the C-identity \eqref{Eq:C} satisfies both
alternative laws. To see this, set $x=1$ in \eqref{Eq:C} to obtain
\eqref{Eq:LA}, and $z=1$ in \eqref{Eq:C} to obtain \eqref{Eq:RA}.

Assume that $M$ is a magma with inverses satisfying \eqref{Eq:C}. Then
\begin{displaymath}
    x^{-1}\cdot xy \refover{Eq:C} x^{-1}\cdot x^{-1}(x\cdot xy)
    \refover{Eq:LA} (x^{-1})^2(x\cdot xy)
    \refover{Eq:C} ((x^{-1})^2x)x\cdot y
    \refover{Eq:RA} (x^{-1})^2 x^2\cdot y.
\end{displaymath}
Therefore, if $M$ satisfies $(x^{-1})^2 = (x^2)^{-1}$, it has the left
inverse property, and thus the inverse property, since \eqref{Eq:C} is
self-dual. We have
\begin{displaymath}
    x^{-1} \refover{Eq:LA} x^{-1}(x\cdot x(x^2)^{-1})
    \refover{Eq:C} (x^{-1}x)x\cdot (x^2)^{-1} =  x(x^2)^{-1},
\end{displaymath}
and hence $(x^{-1})^2$ is equal to
\begin{multline*}
    x^{-1}\cdot x(x^2)^{-1}
    \refover{Eq:RA} x^{-1}\cdot x(x^2\cdot (x^2)^{-1}(x^2)^{-1})
    \refover{Eq:LA} x^{-1}\cdot x(x(x\cdot (x^2)^{-1}(x^2)^{-1}))\\
    \refover{Eq:C} x(x\cdot (x^2)^{-1}(x^2)^{-1})
    \refover{Eq:LA} x^2\cdot (x^2)^{-1}(x^2)^{-1}
    \refover{Eq:RA}(x^2)^{-1}.
\end{multline*}
\end{proof}

This finishes the proof of Theorem \ref{Th:TwoSided}.

\section{Magmas with inverses satisfying an identity of Bol-Moufang
type}

We have seen many examples of identities of Bol-Moufang type. Here is the
general definition: an identity involving one binary operation $\cdot$ is said
to be of \emph{Bol-Moufang type} if it contains three distinct variables, if it
contains three distinct variables on each side, if precisely one of the
variables occurs twice on each side, if all other variables occur once on both
sides, and if the variables are ordered in the same way on both sides.

A systematic notation for identities of Bol-Moufang type was introduced in
\cite{PhVo2004}, according to
\begin{displaymath}
\begin{array}{cc}
    \begin{array}{c|c}
        A&xxyz\\
        B&xyxz\\
        C&xyyz\\
        D&xyzx\\
        E&xyzy\\
        F&xyzz
    \end{array}
    \quad\quad&
    \begin{array}{c|c}
        1&o(o(oo))\\
        2&o((oo)o)\\
        3&(oo)(oo)\\
        4&(o(oo))o\\
        5&((oo)o)o
    \end{array}
\end{array}
\end{displaymath}
For instance, $C25$ is the identity $x((yy)z)=((xy)y)z$. Any identity $Xij$
(with $i<j$) can be dualized to $X'j'i'$ (with $j'<i'$), following
\begin{displaymath}
    A'=F,\quad B'=E,\quad C'=C,\quad D'=D,\quad 1'=5,\quad 2'=4,\quad 3'=3.
\end{displaymath}

The equivalence classes for all identities of Bol-Moufang type \emph{in the
variety of loops} have essentially been determined already in \cite{Fe}, with
the programme completed in \cite{PhVo2004}. With respect to this equivalence we
can often replace identities of Bol-Moufang type by shorter identities, for
instance $x(x\cdot yz) = x(xy\cdot z)$ is equivalent to $x\cdot yz = xy\cdot
z$. Such short, equivalent identities are used in Figure \ref{Fg:BM}.

However, as we have seen while working with Moufang loops, the equivalence
classes do not carry over to magmas with inverses.

For the sake of completeness, we answer (with one exception) the following
question: \emph{Given an identity $I$ of Bol-Moufang type or an identity listed
in Figure \ref{Fg:BM}, is a magma with inverses satisfying $I$ necessarily a
loop?}

The answer is ``yes'' for: all identities equivalent to the C-identity, and all
identities equivalent to the left or right Bol identities. (In all three cases
the equivalence class consists of a single identity, and hence this is just a
restatement of the results in Section \ref{Sc:Eq}.)

The answer is ``no'' for: all identities equivalent to the left, middle, or
right nuclear square identities; all identities equivalent to the flexible
identity; all identities equivalent to the left or right alternative
identities; all identities equivalent to the LC- and RC-identities; and,
perhaps surprisingly, all identities equivalent to the extra identity.

There are $4$ Moufang identities of Bol-Moufang type, and they behave as
described in Section \ref{Sc:Eq}.

The answer is ``yes'' for the following identities equivalent to the
associative law: A24, A25, B34, B35, E13, E23, F14, F24.

The answer is ``no'' for the following identities equivalent to the associative
law: A12, A23, B12, B13, B24, C13, C23, C34, C35, D12, D13, D14, D25, D35, D45,
E24, E35, E45, F34, F45.

All omitted proofs and counterexamples are easy to obtain with \texttt{Prover9}
and \texttt{Mace} \cite{Mc}. Each proof takes only a fraction of a second to
find with a 2GHz processor, and the counterexamples are of order at most $6$.

We have accounted for all identities of Bol-Moufang type, except for B25 and
its dual E14, both of which are equivalent to the associative law in the
variety of loops.

\begin{problem}
Is every magma with inverses satisfying $x((yx)z) = ((xy)x)z$ a group?
\end{problem}

We were not able to resolve the problem despite devoting several days of
computer search to it. If a counterexample exists, it is of order at least
$14$. We would not be surprised to see that the problem holds for all finite
magmas but fails in the infinite case.

\section{One-sided neutral element and one-sided inverses}\label{Sc:OneSided}

Let $Q$ be a groupoid. An element $1\in Q$ is said to be a \emph{left}
(\emph{right}) \emph{neutral element} if $1\cdot x = x$ ($x\cdot 1=x$) holds
for every $x\in Q$. Given a possibly one-sided neutral element $1$, we say that
$x'$ is a \emph{left} (\emph{right}) \emph{inverse} of $x$ if $x'\cdot x=1$
($x\cdot x'=1$).

While defining groups in \cite[p. 4]{Hall}, Marshall Hall remarks that
associative groupoids with a right neutral element and with right inverses are
already groups. He points to \cite{Mann} for a discussion of associative
groupoids with a right neutral element and left inverses that are not groups.
Such groupoids are easy to find:

\begin{example}
Define multiplication on $Q=\{a,b\}$ by $xy = x$, and note that $Q$ is
associative. Moreover, $a$ is a right neutral element and $Q$ has left inverses
with respect to $a$. (By symmetry, $b$ is also a right neutral element and $Q$
has left inverses with respect to $b$.) But $Q$ does not have a two-sided
neutral element and hence is not a group.

Since associativity implies all identities of Bol-Moufang type, $Q$ also shows
that a groupoid satisfying an identity of Bol-Moufang type with a right neutral
element and left inverses is not necessarily a loop.
\end{example}

It view of Hall's remark, it is natural to ask whether Theorem
\ref{Th:TwoSided} can be analogously strengthened. We have:

\begin{theorem}\label{Th:OneSided} Let $Q$ be a groupoid with a left neutral
element and left inverses satisfying one of \eqref{Eq:LB}, \eqref{Eq:M1},
\eqref{Eq:M2}, \eqref{Eq:C}. Then $Q$ is a loop.
\end{theorem}

\begin{proof}
Thanks to Theorem \ref{Th:TwoSided}, it suffices to show that $Q$ has a
two-sided neutral element and two-sided inverses. This is once again easily
accomplished with \texttt{Prover9}. Here is a human proof for the identity
\eqref{Eq:M2}:

Assume that $1x=x$ and $x'x=1$ for every $x\in Q$. With $x=1$, \eqref{Eq:M2}
yields the flexible law. By \eqref{Eq:M2} and flexibility, $yx = (x'x)y\cdot x
= x'\cdot x(yx) = x'\cdot (xy)x$. Using $x'$ instead of $x$ and $x$ instead of
$y$ in the last equality, we deduce $xx' = x''\cdot (x'x)x' = x''\cdot 1x' =
x''x'= 1$, so the inverses are two-sided. Then $x=1x = (xx')x = x(x'x) = x1$,
and the left neutral element $1$ is two-sided, too.
\end{proof}

Let
\begin{equation}\label{Eq:RB}
    x(yz\cdot y) = (xy\cdot z)y\tag{RB}
\end{equation}
be the right Bol identity, the dual to \eqref{Eq:LB}. Since \eqref{Eq:M1} is
dual to \eqref{Eq:M2} and \eqref{Eq:C} is self-dual, we have:

\begin{corollary}\label{Cr:Right}
Let $Q$ be a groupoid with a right neutral element and right inverses
satisfying one of \eqref{Eq:RB}, \eqref{Eq:M1}, \eqref{Eq:M2}, \eqref{Eq:C}.
Then $Q$ is a loop.
\end{corollary}

The left Bol identity \eqref{Eq:LB} cannot be added to the list of identities
in Corollary \ref{Cr:Right}, as the following example shows.

\begin{example} Consider this groupoid:
\begin{displaymath}
    \begin{array}{c|ccc}
        &0&1&2\\
        \hline
        0&0&2&1\\
        1&1&0&2\\
        2&2&1&0
    \end{array}
\end{displaymath}
It has a right neutral element, two-sided inverses, satisfies \eqref{Eq:LB},
but it is not a loop.
\end{example}

However, we have:

\begin{theorem}
Let $Q$ be a groupoid with a two-sided neutral element and right inverses. If
$Q$ satisfies \eqref{Eq:LB} then $Q$ is a loop.
\end{theorem}
\begin{proof}
Let $x1=1x=x$, $xx'=1$ for all $x\in Q$. By \eqref{Eq:LB}, $x'x = x'(x1) =
x'(x\cdot x'x'') = (x'\cdot xx')x''$. Now, $x'' = (x'1)' = (x'\cdot xx')'$, and
thus the previous equality yields $x'x = (x'\cdot xx')(x'\cdot xx')' = 1$. We
are done by Theorem \ref{Th:Bol}.
\end{proof}

For related results on left loops, we refer the reader to \cite{SharmaI},
\cite{SharmaII}. As for the question \emph{If a quasigroup satisfies a given
identity of Bol-Moufang type, is it a loop?}, see \cite{KunenI} and
\cite{PhVoQuasi}.

\section{Acknowledgement}

\noindent Our investigations were aided by the equational reasoning tool
\texttt{Prover9}. Proofs of Theorems \ref{Th:Bol}, \ref{Th:Moufang}, \ref{Th:C}
presented here are significantly simpler than those produced by
\texttt{Prover9}, however. We thank Michael Kinyon for useful discussions, and
for bringing Theorem \ref{Th:Bol} to our attention. We also thank the anonymous
referee for asking about one-sided inverses and neutral elements. Section
\ref{Sc:OneSided} was written in response to his/her inquiry.

%%%%%%%%%%%%%%%%%%%%%%%%%%%%%%%%%%%%%%%%%%%%%%%%%%%%%%%%%%%%%%%%%%%%%%%%%%%%%%%
% BIBLIOGRAPHY                                                                %
%%%%%%%%%%%%%%%%%%%%%%%%%%%%%%%%%%%%%%%%%%%%%%%%%%%%%%%%%%%%%%%%%%%%%%%%%%%%%%%

\bibliographystyle{plain}

\end{document}